\documentclass[12pt]{article}
\usepackage{graphicx}
\usepackage{amsmath}
\usepackage{amssymb}
\usepackage{amsthm}
\usepackage{amsfonts}

\usepackage{url}

\input epsf
\input pstricks
\input pst-node
\input pst-plot
\input pst-eps

\def\bundle#1#2#3{#2\buildrel{#1}\over\longrightarrow #3}
\def\bydef{{\buildrel \rm def \over =}}
\newcommand{\Diff}{\hbox{Diff}}
\newcommand{\Free}{\hbox{Free}}
\newcommand{\Imm}{\hbox{Imm}}
\newcommand{\Hol}{\hbox{Hol}}
\newcommand{\codim}{\hbox{codim\ }}
\newcommand{\rk}{\hbox{rk\ }}

\newcommand{\cA}{{\cal A}}
\newcommand{\cD}{{\cal D}}
\newcommand{\cE}{{\cal E}}
\newcommand{\cF}{{\cal F}}
\newcommand{\cG}{{\cal G}}
\newcommand{\cH}{{\cal H}}
\newcommand{\cL}{{\cal L}}
\newcommand{\cM}{{\cal M}}

\newcommand{\cR}{{\cal R}}

\newcommand{\cU}{{\cal U}}
\newcommand{\bR}{{\mathbb R}}

\newcommand{\bT}{{\mathbb T}}

\newcommand{\R}{{\mathbb R}}

\newcommand{\Rt}{{\mathbb R^2}}
\newcommand{\Rq}{{\mathbb R^q}}
\newcommand{\Rn}{{\mathbb R^n}}

\newtheorem*{theorem*}{Theorem}
\newtheorem{theorem}{Theorem}
\newtheorem{lemma}{Lemma}
\newtheorem{example}{Example}
\newtheorem{remark}{Remark}
\newtheorem{corollary}{Corollary}

\newtheorem*{conjecture*}{Conjecture}
\newtheorem{definition}{Definition}
\newtheorem{proposition}{Proposition}

\title{%
  Proof of a Gromov conjecture\\ 
  on the infinitesimal invertibility\\ 
  of the metric inducing operators
}
\author{Roberto De Leo\\ \small Howard University, Washington DC 20059 (USA)\\ \small INFN, Cagliari (Italy)}

\begin{document}
\maketitle
\begin{abstract}
  We prove a conjecture of Gromov about non-free isometric immersions.
\end{abstract}  
\section{Introduction, main result and notations.}
Let $M$ be a $n$-dimensional manifold, $\Imm^1(M,\Rq)$ the set of $C^1$ immersions of $M$ into $\Rq$, 
$\cG^0(M)$ the bundle of $C^0$ Riemannian metrics over $M$ and $e_q$ the Euclidean metric on $\Rq$. 
Recall that a $C^2$ map $f$ is {\em free} when,
with respect to any coordinate system on $M$ and any frame on $\Rq$, the $n\times q$ matrix $D^2f$ of the first and second
partial derivatives of $f$ has rank $n+s_n$, $s_n=n(n+1)/2$, at every point. More generally, we call {\em 2-rank} of $f$ 
at $x$ the rank of the matrix $D^2f$ at $x$ -- e.g. a free map has 2-rank $n+s_n$ at every point.

It is well known (e.g. see~\cite{Gro86}, Sec.~2.3.1) that the metric inducing operator $\cD_{M,q}:\Imm^1(M,\Rq)\to\cG^0(M)$, 
defined by $\cD_{M,q}(f)=f^*e_q$, is an open map, in the Withney strong topology (throughout the present article we will
always use this topology for our functional spaces), over the (open) set of smooth free maps 
$\Free^\infty(M,\Rq)$. 
For example this means that, if $f_0\in\Free^\infty(M,\Rq)$ and $g_0=\cD_{M,q}(f_0)$ then, for every $g\in\cG^\infty(M)$ close enough 
to $g_0$, there is a map $f$ close enough to $f_0$ such that $g=\cD_{M,q}(f)$. In other words, the solvability of the 
Partial Differential Equation (PDE) 
\begin{equation}
  \label{eq:D(f)=g}
  \cD_{M,q}(f)=g 
\end{equation}
is stable with respect to small smooth perturbations over the (open) subset of the metrics induced by smooth free maps.

In a recent review article (\cite{Gro15}, p. 42) Gromov formulated the following conjecture:
\begin{conjecture*}
    If $q>s_n$, the operator $\cD_{M,q}$ is open over an open dense subset of $C^\infty(M,\Rq)$.
\end{conjecture*}
In \cite{Gro86}, Sec.~2.3.8~(E), Gromov suggested to prove his conjecture when $q\geq n+s_n-\sqrt{n/2}$ by using
the same line of argument of his proof that a generic underdetermined linear Partial Differential Operator (PDO)
admits a right inverse. In~\cite{DL07}, G.~D'Ambra and A.~Loi made a first step in this direction by
showing that $\cD_{\Rt,4}$ is open over a non-empty subset of the (open) subset of maps whose 2-rank is full at every 
point by using ideas from the Lie systems results in~\cite{Gro86}, Sec.~2.3.8~(C). In~\cite{DeL10} we extended this result to 
$\cD_{\Rn,n+s_n-1}$ as a consequence of a general theorem by Duistermaat and Hormander. Here we go back to Gromov's
suggestion and we prove the following result (see Theorems~\ref{thm:infinv} and~\ref{thm:isom} and Corollary~\ref{cor:open}
in Sec.~\ref{sec:mr} for more precise and slightly stronger related statements):
\begin{theorem*}
  If $q\geq n+s_n-\sqrt{n/2}+1/2$, the operator $\cD_{M,q}$ is open over an open dense subset of the set of $C^\infty$ maps 
  $M\to\R^q$ of full 2-rank.
\end{theorem*}
Our proof follows closely the Gromov proof on the right invertibility of generic undetermined linear
PDOs in Section 2.3.8~(E) of~\cite{Gro86}, so we will present in full detail only the parts of the proof 
that we needed to modify or that play some role in our construction.

Throughout the article we will use the following notations for indices: $i=1,\dots,q$, $a=1,\dots,m$, $\alpha,\beta=1,\dots,n$.
Whenever there will be no ambiguity, we will use the Einstein summation convention over repeated indices. Our base manifold will 
be always denoted by $M$, its points by $x$ and its dimension by $n$.
\section{The compatibility PDO $\cL_f$}
In local coordinates $(x^\alpha)$ on $M$ and $y^i$ on $\Rq$, Equation (\ref{eq:D(f)=g}) 
writes as the following second order quadratic PDE of $s_n$ equations in the $q$ variables $f^i$:
\begin{equation}
  \label{eq:nonlin}
  \delta_{ij}\;\partial_\alpha f^i\;\partial_\beta f^j = g_{\alpha\beta}.
\end{equation}
By considering a 1-parametric smooth family of metrics $(g_t)_{\alpha\beta}$ and a corresponding smooth family 
of solutions $f^i_t$, by differentiating with respect to the parameter $t$ we obtain the following {\em linear} 
first order PDE of $s_n$ equations in the $q$ variables $\delta f^i$:
\begin{equation}
  \label{eq:lin}
  2\delta_{ij}\;\partial_\alpha f_0^i\;\partial_\beta \delta f^j = (\delta g_0)_{\alpha\beta},
\end{equation}
where we used the standard notation $\delta=d/dt|_{t=0}$. The $\delta f^j$ are components of 
a map $\delta f\in C^1(M,\Rq)$ that can be thought as a vector tangent to the curve $\{f_t\}\subset C^1(M,\Rq)$ at $f_0$.
%
\begin{definition}
  Let $\Gamma^0\left(S^0_2(M)\right)$ be the set of $C^0$ sections of the tensor bundle of symmetric $(0,2)$ tensors over $M$.
  The PDO 
  $$
  T_{f_0}\cD_{M,q}:T_{f_0}C^1(M,\Rq)\simeq C^1(M,\Rq)\to T_{g_0}\cG^0(M)\simeq\Gamma^0\left(S^0_2(M)\right)
  $$ 
  defined by 
  $$
  T_{f_0}\cD_{M,q}(\delta f)=2\delta_{ij}\;\partial_\alpha f_0^i\;\partial_\beta \delta f^j
  $$ 
  is the tangent map (linearization) of $\cD_{M,q}$ at $f_0$.
  We say that $\cD_{M,q}$ is {\em infinitesimally invertible} over $\cA\subset C^1(M,\Rq)$ if there exists 
  a family $\cE$ of {\em linear} PDOs $\cE_f:\Gamma^s(S^0_2M)\to T_fC^0(M,\Rq)$, $f\in\cA$, 
  of some order $s$ such that:
  \begin{enumerate}
  \item there is an integer $d\geq1$, called {\em defect} of $\cE$, such that $\cA=\cA^d\subset C^d(M,\Rq)$ 
    and $\cA$ is defined by some {\em open} differential relation;
  \item the map $\cE:\cA^d\times\Gamma^s(S^0_2M)\to TC^0(M,\Rq)$ defined by $\cE(f,\eta)\to\cE_f(\eta)$ is a PDO of 
    order $d$ in the first variable and order $s$ in the second;
  \item $T_{f}\cD_{M,q}(\cE_{f}(\delta g))=\delta g$ for all $f\in\cA^d\cap C^{d+1}(M,\Rq)$ and $\delta g\in\Gamma^{s+1}(S^0_2M)$.
  \end{enumerate}
\end{definition}
%
  Among the fundamental results of Nash in his celebrated article~\cite{Nas56} is that $\cD_{M,q}$ admits an infinitesimal
  inverse of order 0 (namely {\em algebraic}) and defect 2 over the set of free maps. This fact is an immediate
  consequence of the product rule for derivatives: indeed, since 
  $\partial_\beta\left(f^i\;\delta f^j\right)=\delta f^j\;\partial_\beta f^i+f^i\;\partial_\beta\delta f^j$, 
  the PDE system (\ref{eq:lin}) is equivalent to the following {\em algebraic} system of $n+s_n$ equation in the $q$ unknowns $\delta f^i$:
  \begin{equation}
    \label{eq:NashTrick}
    \begin{split}
      \delta_{ij}\;\partial_\alpha f_0^i\;\delta f^j &= h_\alpha,\\
      2\delta_{ij}\;\partial^2_{\alpha\beta}f_0^i\;\delta f^j &= \partial_\alpha h_\beta + \partial_\beta h_\alpha - (\delta g_0)_{\alpha\beta},
    \end{split}
  \end{equation}
  where the $h_\alpha$ are $n$ auxiliary functions. When $f_0$ is a free map, by definition the matrix of its
  first and second derivative is surjective at every point $x\in M$ and so system (\ref{eq:NashTrick}) is solvable 
  {\em independently on the choice of the $h_\alpha$} (which, therefore, are usually all set equal to zero).

The following Implicit Function Theorem in infinite dimension, a powerful Gromov's generalization of the celebrated 
work of Nash on isometric immersions, allow us to focus, rather than on the original quadratic PDE system~(\ref{eq:nonlin}), 
on the (much simpler) linear system~(\ref{eq:NashTrick}):
\begin{theorem}[Gromov, 1986]
  \label{thm:IFT}
  Let $F\to E$ be a fiber bundle, $G\to E$ a vector bundle and $\Gamma^rF$ and $\Gamma^sG$, respectively, 
  the sets of their $C^r$ and $C^s$ sections.
  If a PDO of order $r$, $\cD_r:\Gamma^rF\to\Gamma^0G$, admits an infinitesimal inversion of order $s$ and defect $d$ over $\cA^d\subset\Gamma^dF$, 
  then for every $f_0\in\cA^d\cap\Gamma^\infty F$ there exists a neighborhood of zero $\cU\subset \Gamma^{\bar s+s+1}G$,
  $\bar s=\max\{d,2r+s\}$, such that, for every $g\in\cU\cap\Gamma^{\sigma+s}G$, $\sigma\geq \bar s+1$, the
  equation $\cD(f)=\cD(f_0)+g$ has a $C^\sigma$ solution.
\end{theorem}
\begin{corollary}[Gromov, 1986]
  Under the hypotheses of the theorem above, the restriction of $\cD_r$ to $\cA^\infty=\cA^d\cap\Gamma^\infty E$ is an open map.
\end{corollary} 
In case of the metric inducing operator $\cD_{M,q}$, this implies the following celebrated results of Nash:
\begin{theorem}[Nash, 1956]
  If $f_0\in Free^\infty(M,\Rq)$ and $g_0=\cD_{M,q}(f_0)$, then the $C^\sigma$ metric $g_0+g$ can be induced by 
  a $C^\sigma$ immersion of $M$ into $\Rq$ for all $g$ that are ``$C^\sigma$-small enough'' and every $\sigma\geq3$. 
\end{theorem}
Clearly this result is void when $Free^\infty(M,\Rq)$ is empty, in particular for $q<n+s_n$. Nevertheless, in the interval 
$n+s_n>q>s_n$, although no free map can arise, there can still be maps whose 2-rank is full at every point.
\begin{example}
  \label{ex:frEuc}
  Recall that the maps $f_n:\R^n\to\R^{n+s_n}$ defined by $$f_n(x^1,\dots,x^n)=(x^1,\dots,x^n,(x^1)^2,x^1x^2,\dots,(x^n)^2)$$
  are free for all $n=1,2,\dots$. The maps $f_{n,m}:\R^n\to\R^{n+s_n-m}$, $m=1,\dots,n$, obtained by composing $f_n$ with any projection $\pi_{n,m}:\R^{n+s_n}\to\R^{n+s_n-m}$
  that ``drops'' any $m$ of the components of index larger than $n$ are all maps of full 2-rank.
\end{example}
\begin{example}
  \label{ex:frTor}
  The maps $f_n:\bT^n\to\bR^{2n}$ defined by
  $$
  f_n(x^1,\dots,x^n)=(\sin x^1,\cos x^1,\dots,\sin x^n,\cos x^n)
  $$
  are all full 2-rank maps, with $m=s_n-n$. In particular there are full 2-rank maps from $\bT^2$ to $\R^4$. 
  We recall that, on the contrary, it is still unknown whether there exist free maps from $\bT^2$ to $\R^5$ (see~\cite{Gro15}, Sec.~2.2). 
\end{example}
Now let us set $m=n+s_n-q$ and let $f_0$ be a full 2-rank map $M\to\R^q$. Then, among the $n+s_n$ vectors $(\partial_\alpha f_0,\partial^2_{\alpha\beta}f_0)$,
$\alpha\leq\beta$, at every point $x\in M$ hold exactly $m$ non-trivial linear relations
$$
\lambda^\alpha_a(f_0,x) \partial_\alpha f_0(x) + 
\sum_{\alpha\leq\beta}\lambda^{\alpha\beta}_a(f_0,x)\partial^2_{\alpha\beta}f_0(x)=0\,,\;\;\;a=1,\dots,m,
$$
where the $m(n+s_n)$ coefficients $(\lambda_a^\alpha(f_0),\lambda^{\alpha\beta}_a(f_0))$, $\alpha\leq\beta$, 
can be chosen as polynomial in the components of $(\partial_\alpha f_0,\partial^2_{\alpha\beta}f_0)$.

A solution of system~(\ref{eq:NashTrick}) therefore exists if and only if the same relation holds on the right hand side
of the system, namely if the PDE
\begin{equation}
  \label{eq:Lf(h)=g}
  \lambda^\alpha_a(f_0)\, h_\alpha + \frac{1}{2}\sum_{\alpha\leq\beta}\lambda^{\alpha\beta}_a(f_0)(\partial_\beta h_\alpha+\partial_\alpha h_\beta)
  = 
  \frac{1}{2}\sum_{\alpha\leq\beta}\lambda^{\alpha\beta}_a(f_0)\,(\delta g_0)_{\alpha\beta}.
\end{equation}
in the $n$ unknowns $h_\alpha$ does admit a solution. Equivalently, a sufficient condition for the solvability of system~(\ref{eq:NashTrick})
is the surjectivity of the {\em compatibility 1-st order linear PDO} $\cL_{f}:C^1(M,\R^n)\to C^0(M,\R^m)$ defined by
$$
\cL_{f}(h_1,\dots,h_n) = 
\lambda^\alpha_a(f)\,h_\alpha + \hat\lambda^{\alpha\beta}_a(f)\,\partial_\beta h_\alpha,
$$
where $\hat\lambda^{\alpha\alpha}_a=\lambda^{\alpha\alpha}_a$ and $\hat\lambda^{\alpha\beta}_a=\hat\lambda^{\beta\alpha}_a=\frac{1}{2}\lambda^{\alpha\beta}_a$ for $\alpha<\beta$.

In the next sections we study general conditions that ensure the surjectivity of this class of PDOs.
%
\section{Surjectivity of generic underdetermined PDOs}
In~\cite{Gro86}, Sec.~2.3.8, Gromov generalizes Nash's technique of purely algebraic inversion of the linear PDO $T_f\cD_{M,\Rq}$ 
to any underdeterminate linear PDO. The price of this generalization is that the order of the inverse is 
usually far from being zero. 

In order to state the main results, we need first to introduce a series of definitions and notations. 
Let $F\to E$ and $G\to E$ be two vector bundles over $E$ with fiber dimension equal to $q$ and $q'$ respectively.
We denote by $\Gamma^r F$ the set of its $C^r$ sections and by $J^rF$ the affine bundle of the $r$-jets of its sections,
and similarly for $G$.
\begin{definition}
  A linear PDO of order $r$ on $F$ with values in $G$ is given by either of these equivalent definitions:
  \begin{enumerate}
  \item a linear map $\cL_r:\Gamma^rF\to\Gamma^0G$ given, in coordinates, by
    $$
    \cL_r(f) = \left(\sum_{|A|\leq r}\Lambda^{aA}_i\partial_Af^i\right)
    =\left(\Lambda^{a}_if^i+\Lambda^{a\alpha}_i\partial_\alpha f^i+\dots+\Lambda^{a\alpha_1\dots\alpha_r}_i\partial_{\alpha_1\dots\alpha_r}f^i\right)
    $$
  \item a smooth section $\Lambda_r$ of the bundle $Hom(J^rF,G)\to E$ of all linear homomorphisms between $J^rF$ and $G$, namely 
    $$
    \Lambda_r(x)=\left(x,\Lambda^{a}_i(x),\dots,\Lambda^{a\alpha_1\dots\alpha_r}_i(x)\right)
    $$
  \end{enumerate}
  The relation between the two definitions is that $(j^rf)^*\Lambda_r = \cL_r(f)$ for every $f\in\Gamma^rF$.
\end{definition}
\begin{example}
  \label{ex:Lxi}
  The simplest linear PDOs over a manifold $M$ are associated to vector fields 
  $\xi\in\Gamma^\infty(TM)$, i.e. a smooth sections of the tangent bundle 
  $\bundle{\tau_M}{TM}{M}$. We denote the corresponding first-order linear 
  homogeneous PDOs by $L_\xi:C^1(M)\to C^0(M)$ (Lie derivative in the $\xi$ direction),
  namely here $E=M$ and $F=G=M\times\bR$.
  In coordinates $L_\xi=\xi^\alpha\partial_\alpha$ and the corresponding map 
  $\Lambda_\xi:J^1(M,\bR)\to \bR$ is defined as $\Lambda_\xi(x^\beta,f,f_\beta)=\xi^\alpha(x^\beta)f_\alpha$,
  namely $(\Lambda_\xi)^1_1=0\,,\;(\Lambda_\xi)^{1\alpha}_1=\xi^\alpha$.
  The corresponding PDE $L_\xi(f)=g$ is called {\em cohomological equation}, whose solvability has been
  recently studied in several contexts (see~\cite{For99,Nov08,DeL15c}).
\end{example}
\begin{remark}
  It was shown in~\cite{DeL11} that, in case of vector fields with no zeros on $\R^2$, $L_\xi$ is surjective
  if and only if $\xi$ is conjugate to a constant vector field (recall that on the 2-torus even this condition is not enough
  for the surjectivity, e.g. see\cite{For95}). On the contrary, in~\cite{Gro86}, Sec 2.3.8~(C), Gromov
  has shown that, for any manifold $M$, the operator $(L_{\xi_1},L_{\xi_2}):C^1(M,\R^2)\to C^0(M)$ defined by 
  $(L_{\xi_1},L_{\xi_2})(f_1,f_2)=L_{\xi_1}(f_1)+L_{\xi_2}(f_2)$ is surjective for any pair of vector fields $\xi_1,\xi_2\in\Gamma^\infty(TM)$ 
  in general position. In other words, topology affects the solvability of linear first order PDEs only 
  in the equal dimension case $q=q'$
\end{remark}
Now consider an open subset $\cH\subset J^{r+s}\left(Hom(J^rF,G)\right)$ for some $s\geq0$. 
\begin{definition}
  A $\cH$-universal right inverse for a linear PDO $\cL_r$ of order $r$ such that $\Lambda_r(E)\subset\cH$
  is a PDO $\cM_s:\Gamma^\infty\cH\times\Gamma^\infty G\to\Gamma^\infty F$
  of order $s+r$ in the first component and linear of order $s$ in the second such that
  $$
  \cL_r(\cM_s(\cL_r,g))=g\,,\;\hbox{for all}\;g\in\Gamma^\infty G.
  $$
\end{definition}
\begin{example}
  \label{ex:large}
  Let $\xi_1,\dots,\xi_q\in\Gamma^\infty(TM)$, $f\in C^1(M,\R^q)$, $q>n$, and define the linear PDO 
  $\Xi:C^1(M,\R^q)\to C^0(M)$ as $\Xi(f)=L_{\xi_1}(f^1)+\dots+L_{\xi_q}(f^q)$, where $L_\xi(f)$ is the Lie derivative 
  of $f$ with respect to $\xi$. 
  We say that $\Xi$ is 
  \emph{large} if the $q\times(n+1)$ matrix 
  \begin{equation}
    \begin{pmatrix}
      \xi^1_1&\cdots&\xi^n_1&-\partial_\alpha\xi^\alpha_1\cr
      \vdots&\vdots&\vdots&\vdots\cr
      \xi^1_q&\cdots&\xi^n_q&-\partial_\alpha\xi^\alpha_q\cr
    \end{pmatrix}
  \end{equation}
  has full rank at every point in any (and so every) coordinate system. Now, let $\cH$ be the open subbundle of 
  $J^1\left(Hom(J^1(M,\R^q),J^0(M,\R))\right)$ spanned by the 1-jets of sections of large PDOs on $C^1(M,\R^q)$ and let
  $(\lambda^1(\Xi),\dots,\lambda^q(\Xi))$ be the solution of the system
  $\lambda^i\xi_i^\alpha=0,\lambda^i\partial_\alpha\xi_i^\alpha=-1$ closest to the origin. Then the PDO
  $\cM_0:\Gamma^\infty\cH\times C^\infty(M)\to C^\infty(M,\R^q)$ defined by $\cM_0(\Xi,g)=(\lambda^i(\Xi)g)$
  is a $\cH$-universal right inverse for $\Xi$ of order 0 since it depends only algebraically on $g$ and
  $\Xi(\lambda^i(\Xi)g)= \xi_i(\lambda^i(\Xi)g)=\xi_i^\alpha\partial_\alpha\lambda^i(\Xi)g +\lambda^i(\Xi)\xi_i^\alpha\partial_\alpha g=g$.
\end{example}
\begin{example}
  The technique of Nash to solve the PDE system~(\ref{eq:lin}) through the algebraic system~(\ref{eq:NashTrick})
  shows that the linearization of the isometric operator $T_f\cD_{M,q}$ admits a 0-order inverse
  $\cM_0:\Gamma^\infty\cH\times \Gamma^\infty(S^0_2M)\to C^\infty(M,\R^q)$, 
  where $\cH$ is the subbundle of $J^1(Hom(J^1(M,\R^q),J^0(M,\R)))$
  spanned by the 1-jets of the operators $T_{f}\cD_{M,q}$ with $f$ free.
  In this case, indeed, the matrix of first and second partial derivatives of $f$ is {\em surjective} and
  the value of $\cM_0(f,\delta g)$ can be defined, for example, as the solution of~(\ref{eq:NashTrick}) closer 
  to the origin. Note that for $f$ free the $h_\alpha$ play no role and we can set all of them to zero.
\end{example}
%
%
\begin{definition}
  The {\em formal adjoint} of a linear PDO $\cL_r:\Gamma^rF\to\Gamma^0G$ of order $r$ is the linear PDO 
  $\overline{\cL_r}:\Gamma^rG\to\Gamma^0F$ of order $r$ defined by
  $$
  \overline{\cL_r}(g) = \left(\sum_{|A|\leq r} \overline{\Lambda}^{iA}_a\partial_A g^a\right) \bydef \left(\sum_{|A|\leq r}(-1)^{|A|}\partial_A\left(\left(\overline{\Lambda^i_a}\right)^{A}g^a\right)\right),
  $$
  where, for every multi-index $A$, the matrix $\left(\overline{\Lambda^i_a}\right)^A$ is the transpose of the matrix $(\Lambda^{aA}_i)$.
\end{definition}
\begin{example}
  The formal adjoint of the linear first-order PDO $\Xi(f^1,\dots,f^q)=\xi^{\alpha}_i\partial_\alpha f^i$ in Example~\ref{ex:large}
  is given by $\overline{\Xi}(g) = -\partial_\alpha\left(\xi^\alpha_i g\right) = -(\partial_\alpha\xi_i^\alpha)g - \xi^\alpha_i \partial_\alpha g$.
\end{example}
Note that the adjunction itself is a PDO $J^r(Hom(J^rF,G))\to Hom(J^rG,F)$ that satisfies $\overline{\overline{\cL_r}}=\cL_r$ and
$\overline{\cL_r\cM_s} = \overline{\cM_s}\;\overline{\cL_r}$.
Gromov's crucial observation, generalizing Nash's method of solving (\ref{eq:lin}) through (\ref{eq:NashTrick}), 
is that, while the problem of finding an $\cH$-universal right inverse 
$\cM_s:\Gamma^\infty\cH\times\Gamma^\infty G\to\Gamma^\infty F$ for $\cL_r:\Gamma^r F\to\Gamma^0 G$ 
requires studying the solvability of a (usually non-trivial) linear PDE of order $r$ in the components of $\cM_s$, namely
\begin{equation}
  \label{eq:LM}
  \sum_{|A|\leq r}\Lambda^{aA}_i\partial_A\left(\sum_{|B|\leq s}M^{iB}_b\partial_B\right) = \delta^a_b,
\end{equation}
the equivalent problem of findind a {\em left} $\cH$-universal inverse $\overline{\cM_s}:\Gamma^\infty\cH\times\Gamma^\infty F\to\Gamma^\infty G$ 
for the adjoint PDO $\overline{\cL_r}:\Gamma^r G\to\Gamma^0 F$, namely 
\begin{equation}
  \label{eq:ML}
  \sum_{|B|\leq s}\overline{M}^{aB}_i\partial_B\left(\sum_{|A|\leq r}\overline{\Lambda}^{iA}_b\partial_A\right) = \delta^a_b,
\end{equation}
involves just solving a purely {\em algebraic} linear system in the $s$-jets of $\overline{\Lambda}_r$, namely in the $(s+r)$-jets of $\Lambda_r$, 
and therefore it can be solved by just combinatorial and transversality arguments. Following this idea, in~\cite{Gro86} Gromov was able
to prove the following:
\begin{theorem}[Gromov, 1986]
  \label{thm:linPDOs}
  If $q>q'$, a generic linear PDO $\cL_r:\Gamma^r F\to\Gamma^0 G$ is surjective.
\end{theorem}
%
%
%
\section{Surjectivity of Upper Totally Symmetric PDOs}
%
As it often happens for very general theorems, the proof of Theorem~\ref{thm:linPDOs} can be used, with minor changes, 
to prove several particular cases not covered by the general statement. In this section we adapt Gromov's proof to a class 
of PDOs that plays a crucial role in the proof of our main theorem.
\begin{definition}
  \label{def:UTS}
  We call a linear PDO $\cL_r:C^r(M,\R^{m+1})\to C^0(M,\R^m)$, with
  $$\cL_r(h_1,\dots,h_{m+1})=\left(\Lambda^{\alpha}_a h_\alpha+\Lambda^{\alpha\beta_1}_a\partial_{\beta_1} h_\alpha+\dots
  +\Lambda^{\alpha\beta_1\dots\beta_r}_a\partial_{\beta_1\dots\beta_r}h_\alpha\right)\,,$$
  \emph{upper totally symmetric} (UTS) if its highest order components $\Lambda^{\alpha\beta_1\dots\beta_r}_a$ are symmetric 
  with respect to all permutations of the indices $\alpha\beta_1\dots\beta_r$ that keep the first one not larger than 
  $m+1$. We denote by $\cU^r_m(M)$ the subbundle of $Hom(J^r(M,\R^{m+1}),J^0(M,\R^m))$ 
  whose fibers are spanned, at every $x\in M$, by the images of germs of UTS PDOs about $x$.
\end{definition}  
Note that, in a UTS PDO, the dimension of the set of terms or highest order decreases from $m(m+1){n\choose r}$ to $m\sum_{i=0}^m {n-i\choose r}$.

We first show, with a purely combinatorial argument, that $\cL_r$ does admit a {\em formal} universal right inverse: 
\begin{lemma}
  There exists a $s_{n,m}$ such that, for all $s\geq s_{n,m}$, it exists an open dense {\em algebraic} subbundle $\cH^{r,s}_m(M)\subset J^{r+s}\cU^{r}_m(M)$ 
  with the property that every UTS PDO $\Lambda_r:M\to Hom(J^r(M,\R^{m+1}),J^0(M,\R^m))$ satisfying $j^{r+s}\Lambda_r(M)\subset\cH^{r,s}_m(M)$ 
  admits a $\cH^{r,s}_m(M)$-universal right inverse.
\end{lemma}
\begin{proof}
  As mentioned in the previous section, finding a right inverse for $\cL_r$ is equivalent to solving the algebraic system (\ref{eq:ML}).
  In every fiber of $J^{s}\cU^r_m(M)$, system (\ref{eq:ML}) writes as
  %
  \begin{equation} \label{eq:hs2}
    \begin{cases}
      \displaystyle\sum_{|A|\leq s}\overline{M}^{aA}_\alpha\overline{\Lambda}^\alpha_{bA}=\delta^a_{b}\cr
      \noalign{\medskip}
      \displaystyle\sum_{|A|\leq s}\overline{M}^{aA}_\alpha\overline{\Lambda}^{\alpha\beta_1}_{bA}
      +\displaystyle\sum_{|A|\leq s-1}\overline{M}^{a\beta_1A}_\alpha\overline{\Lambda}^\alpha_{bA}=0\cr
      \noalign{\medskip}
      \displaystyle\sum_{|A|\leq s}\overline{M}^{aA}_\alpha\overline{\Lambda}^{\alpha\beta_1\beta_2}_{bA}
      +\displaystyle\sum_{|A|\leq s-1}\overline{M}^{a\beta_1A}_\alpha\overline{\Lambda}^{\alpha\beta_2}_{bA}
      +\displaystyle\sum_{|A|\leq s-2}\overline{M}^{a\beta_1\beta_2A}_\alpha\overline{\Lambda}^\alpha_{bA}=0\cr
      \noalign{\medskip}
      \vdots\cr
      \displaystyle\sum_{|A|\leq1}\overline{M}^{a\beta_1\dots\beta_{s-1}A}_\alpha\overline{\Lambda}^{\alpha\beta_s\dots\beta_{s+r-1}}_{bA}
      +\overline{M}^{a\beta_1\dots\beta_{s}}_\alpha\overline{\Lambda}^{\alpha\beta_{s+1}\dots\beta_{s+r-1}}_b=0\cr
      \noalign{\medskip}
      \overline{M}^{a\beta_1\dots\beta_s}_\alpha\overline{\Lambda}^{\alpha\beta_{s+1}\dots\beta_{s+r}}_b=0\cr
    \end{cases}
  \end{equation}
  where $1\leq a,b\leq m$, $1\leq\alpha\leq m+1$ and $1\leq\beta_\ell\leq n$ for all $1\leq\ell\leq s+r$.

  Note that (\ref{eq:hs2}) naturally splits in $m$ independent systems, one for each value of $a$, of $(m+1){n+s\choose s}$ unknowns $\overline{M}^{aA}_\alpha$  
  in $m{n+r+s\choose r+s}$ equations. All rows but one in each of these independent systems are homogeneous and so each of them admits a {\em formal} 
  algebraic solution $\overline{M}^{aA}_\alpha(\overline{\Lambda}^{\alpha B}_{bA})$
  if and only if the non-homogeneous row is {\em not} a linear combination of the homogeneous ones with coefficient in the ring of rational functions in the fiber
  coordinates. Since the system is {\em triangular}, in the sense that in the column of each $\overline{M}^{aA}_\alpha$ only one coefficient $\overline{\Lambda}^{\alpha}_{bA}$ 
  of zero order appears (and it does it in the non-homogeneous row), we can recursively substitute the expressions of such coefficients in the other columns
  so that in the end no $\overline{\Lambda}^{\alpha}_{bA}$ appears anywhere except in the non-homogeneous row. When that row is a linear combination of the homogeneous
  ones, therefore, we can express the coefficients $\overline{\Lambda}^{\alpha}_{bA}$ as rational functions of all other coefficients $\overline{\Lambda}^{\alpha B}_{bA}$, $|B|\geq1$,
  and of a number of rational functions of all $\overline{\Lambda}^{\alpha B}_{bA}$ (the coefficient of the linear combination) equal to the number of homogeneous rows. 
  Since the $(m+1){n+s\choose s}$ coefficients are independent
  coordinates, such combination clearly cannot exist if we choose $s$ big enough that $(m+1){n+s\choose s}>m{n+r+s\choose r+s}$.
  Note that this is always possible since 
  $$
  (m+1){n+s\choose s}>m{n+r+s\choose r+s}\hbox{\ \ iff\ \ }\frac{m+1}{m}>\prod_{\ell=1}^r\left(1+\frac{n}{s+\ell}\right)
  $$
  and the r.h.s. term converges to 1 for $s\to\infty$. 

  The $\overline{M}^{aA}_\alpha$ solving (\ref{eq:hs2}) are algebraic functions of the coordinates $\overline{\Lambda}^{\alpha B}_{bA}$. Let $P$ be the polynomial
  obtained as the product of all denominators of such functions and define $\cH=\{P\neq0\}$. Then $\cH^{r,s}_m(M)$ is an algebraic subbundle of
  $J^{r+s}\cU^r_m(M)$ satisfying the claim of the theorem.
\end{proof}
Of course such inverse remains only formal until we are able to prove the existence of global sections $\Lambda_r$ whose $(r+s)$-jets images are contained
in $\cH^{r,s}_m(M)$. Following Gromov, we prove this by showing that the codimension of $\cH^{r,s}_m(M)$ grows with $s$ large enough that the image 
of the $(r+s)$-jets of {\em generic} sections must be entirely contained into it.
\begin{definition}
  \label{def:transversality}
  A linear PDO $\cL_r:\Gamma^rF\to \Gamma^0G$, $\cL_r(f)=\left(\sum_{|A|\leq r}\Lambda^{a A}_i\partial_A f^i\right)$, is {\em transversal} to a 
  codimension-$k$ submanifold $E_0\subset E$ at $x_0\in E_0$, locally defined by $H(x)=0$ 
  for some smooth map $H:E\to\R^k$, if the $k$ matrices $q\times q'$ 
  $$
  T^\ell(x_0)=\left(\Lambda^{a\alpha_1\dots\alpha_r}_i(x_0)\partial_{\alpha_1}H^{\ell}(x_0)\dots\partial_{\alpha_r}H^{\ell}(x_0)\right)
  $$
  have all maximal rank, i.e. $\rk\Lambda^\ell(x_0)=q'$ for all $\ell=1,\dots,k$. If $\cL_r$ is not transversal to $E_0$, we say that it is 
  {\em tangent} to it. We say that $E_0$ is a {\em characteristic submanifold} for $\cL_r$ if $\cL_r$ is tangent to $E_0$ at every point.
\end{definition}  
\begin{example}
  The linear 1-st order PDO $\Xi$ in Example~\ref{ex:large} is transversal to the hypersurface $E_0=\{H(x)=0\}$ at $x=x_0$ if $L_{\xi_i}H(x_0)\neq0$
  for at least one index $1\leq i\leq q$, i.e. if at least one vector field $\xi_i$ is transversal to $E_0$ at $x_0$ in the usual sense (namely
  that $\xi_i(x_0)$ is not contained in $T_{x_0}E_0$). On the contrary, $E_0$ is characteristic for $\Xi$ if the characteristics (i.e. the integral
  trajectories) of all vector fields $\xi_i$ are entirely contained in $E_0$.
\end{example}
\begin{example}
  Consider a (pseudo-)Riemannian constant metric $g=g_{\alpha\beta}dx^\alpha dx^\beta$ on $\R^n$, so that the Laplacian PDO is given
  by $\Delta_g(f)=g^{\alpha\beta}\partial^2_{\alpha\beta}f$, where $(g^{\alpha\beta})$ is the inverse matrix of $(g_{\alpha\beta})$. Then
  $\Delta_g$ is transversal to the hypersurface $E_0=\{H(x)=0\}$ at $x=x_0$ iff $g(dH,dH)=0$ is zero at $x=x_0$. For example, if
  $g$ is Riemannian then it has no characteristic submanifold of codimension 1 while, if its signature is $(1,n-1)$, its only characteristic
  codimension-1 submanifolds are its light-cones.
\end{example}
%
%
\begin{lemma}
  \label{lem:UTS}
  Generic UTS PDOs have no characteristic submanifold of positive codimension.
\end{lemma}
\begin{proof}
  Assume that $E_0=\{H(x)=0\}$, $H:E\to\R^k$, is a codimension-$k$ characteristic submanifold for $\cL_r:C^r(M,\R^{m+1})\to C^0(M,\R^m)$. 
  Then $E_0$ is contained in the set $C_k=\{\rk T^\ell(x)<m\,,\ell=1,\dots,k\}$. The number of such functionally independent relations can
  be evaluated by considering the case $H^\ell(x^1,\dots,x^n)=x^\ell$, $1\leq\ell\leq k$. When $r>1$, the $m\times(m+1)$ matrices 
  $$
  T^\ell(x)=(\Lambda^{\alpha\overbrace{\ell\dots\ell}^r}_a(x))
  $$
  have no element in common for all $\ell$ and so $\codim C_k=2k$, which is incompatible with the facts that $\codim E_0=k$ and $E_0\subset C_k$.
  When $r=1$ and $k>m+1\geq2$ then the number of independent relations is $2k-1$, so we get again an absurd, while for $k\leq m+1$ we have
  again $\codim C_k=2k$.
\end{proof}
%
%
%
The proof of the following proposition is the same as the one of Lemma 4 in Section~2.3.8 (E) in~\cite{Gro86}.
\begin{theorem}
  \label{thm:UTSrel}
  There exists a $\bar s=s(n,m,r)\geq\frac{n}{(1+1/m)^{1/r}-1}$
  such that, for every $s\geq \bar s$, there is a maximal algebraic subbundle $\cH^{r,s}_m(M)\subset J^{r+s}\cU^r_m(M)$ whose complement has codimension larger than $n$ 
  and with the property that every UTS PDO $\Lambda_r:M\to Hom(J^r(M,\R^{m+1}),J^0(M,\R^m))$ satisfying $j^{r+s}\Lambda_r(M)\subset\cH^{r,s}_m(M)$ admits a 
  $\cH^{r,s}_m(M)$-universal right inverse.
  %
\end{theorem}
As an immediate corollary we get the surjectivity of generic UTS PDOs:
\begin{theorem}
  \label{thm:UTS}
  A generic upper totally symmetric PDO $\cL_r:C^r(M,\R^{m+1})\to C^0(M,\R^m)$ is surjective.
\end{theorem}
\begin{remark}
  \label{rm:sub}
  Note that in no part of the proofs of the Theorems and Lemmas of this section that lead to Theorem~\ref{thm:UTS} were ever 
  explicitly used the components of the terms of intermediate degree $0<r'<r$. Hence all properties proved in this section hold,
  mutatis mutaindis, for PDOs that are sections of any subbundle of $\cU^r_m(M)$ defined by any number of closed relations
  among such components (the same argument holds, indeed, also for Gromov's Theorem~\ref{thm:linPDOs}).
\end{remark}
%
%
\section{Surjectivity of the compatibility operators}
Now we go back, in a slightly more general setting, to the compatibility operators we mentioned in Section 1 and provide condition for their surjectivity.
\begin{definition}
  \label{def:rank}
  We call $r$-rank of a map $f:M\to\R^q$ at $x\in M$ the rank of the matrix $D^r_xf$ of the partial derivatives of $f$ at $x$ up to order $r$. A map is $r$-free 
  if $\rk D^r_xf=s_{n,r}\bydef{n+r\choose r}-1$ for every $x\in M$, namely if the vectors of all its partial derivatives up to order $r$ are linearly independent
  at every point. Finally, we say that a $r$-free map $f:M\to\R^q$ has {\em full $r+1$-rank} when $s_{n,r}<q\leq s_{n,r+1}$ and $\rk D^{r+1}_xf=q$ for all $x\in M$,
  namely when, at every $x$, its $s_{n,r}$ vectors of its partial derivatives up to order $r$ are mutually linearly independent and exactly $q-s_{n,r}$ of 
  the $s_{n,r+1}-s_{n,r}$ vectors of its derivatives of order $r+1$ are linearly independent from them.
\end{definition}  
Note that clearly $r$-free maps are also $r'$-free for all $r'\leq r$ and can only
arise when $q\geq s_{n,r}$. 
\begin{example}
  A 1-free map is an immersion, a 2-free map is a standard free map. A full 2-rank map $f:M\to\R^q$, $n<q\leq n+s_n$ is an immersion
  such that the $n+s_n$ vectors $\{f_\alpha(x),f_{\alpha\beta}(x)\}$ span at every point the whole $\R^q$. A full $(r+1)$-rank map with
  $q=s_{n,r+1}$ is a $(r+1)$-free map.
\end{example}
\begin{remark}
  Except for the case $r=1$, this definition is not invariant with respect to the whole group $\Diff(\R^m)$ but only with respect to its subgroup of {\em affine} 
  transformation. Geometrically, this corresponds to the fact that the only tensor bundle among the $J^{r}(M,\R^m)\to J^{r-1}(M,\R^m)$ is the one with $r=1$, while
  all others are affine. These definitions, though, can be made invariant if we replace $\R^m$ with a Riemannian manifold and partial derivatives with the corresponding
  covariant derivatives.
\end{remark}

We are interested in the following natural map of full $(r+1)$-rank maps into PDOs of order $r$. Given a full $(r+1)$-rank map $f:M\to\R^q$, between
the $s_{n,r+1}$ vectors of its partial derivatives up to order $r+1$ hold at every point $m=s_{n,r+1}-q$ linear relations 
$$
\sum_{|A|\leq r+1}\lambda_a^A(f)\,\partial_A f^i=0,\;a=1,\dots,m
$$
where the $m\cdot s_{n,r+1}$ coefficients $\lambda_a^A(f)\in C^0(M)$, 
defined modulo a factor $\pm e^p$, $p\in C^0(M)$, are not all zero at the same time at any point.
%
\begin{definition}
  \label{def:Lf}
  We call {\em compatibility operator} associated to the full $(r+1)$-rank map $f:M\to\R^q$ the PDO of order $r$ 
  $\cL_f:C^r(M,\R^n)\to C^0(M,\R^m)$ defined as
  $$
  \cL_f(h) = \sum_{|A|\leq r}\lambda_a^{\alpha A}(f)\,\partial_A h_\alpha\,.
  $$
\end{definition}
The elementary linear algebra lemma below shows that the $\lambda_a^A(f)$ can all be chosen as homogeneous fully antisymmetric polynomials of degree $q$.
\begin{definition}
  \label{def:v}
  Given a set of $N$ vectors $v=\{v_1,\dots,v_N\}\subset\R^{N-m}$ , we denote by $V_{\ell}$, $\ell=1,\dots,m$, the $(N-m+1)\times(N-m)$ matrix 
  whose $N-m+1$ columns are the components of the vectors $\{v_1,\dots,v_{N-m},v_{N-m+\ell}\}$, by $V_{\ell,i}$, $i=1,\dots,N-m+1$ the $(N-m)\times(N-m)$ 
  matrix obtained from $V_\ell$ by removing the $i$-th column and by $V_0$ the $(N-m)\times(N-m)$ matrix whose columns are the components of the first $N-m$ vectors. 
\end{definition}
\begin{lemma}
  \label{lem:v}
  Let $\{v_1,\dots,v_N\}$ be a set of $N$ vectors in $\R^{N-m}$ spanning the whole vector space and ordered so that the first $N-m$ are linearly
  independent. Let $\lambda^i_a$ be non-trivial coefficients expressing the linear dependency conditions between the $v_i$, namely 
  $\lambda^i_a v_i=0$ for $a=1,\dots,m$. Then, for every $a$, these coefficients can be chosen as: 
  $\lambda^i_a=(-1)^i\det V_{a,i}$ for $i=1,\dots,N-m$; $\lambda^{N-m+a}_a=(-1)^{N-m+1}\det V_0$; $\lambda^i_a=0$ otherwise.
%
\end{lemma}
In order to ensure the existence of an inverse for $\cL_f$ we need to ensure the differential independence of its components.
Clearly such independence cannot take place when there are more components $\lambda^{\alpha A}_a(f)$ than independent variables $f^i$.
\begin{example}
  Consider a map $f:\R^3\to\R^2$ of full 1-rank (i.e. an immersion), so that the three vectors $f_x,f_y,f_z$ satisfy at every point
  a non-trivial linear dependence condition $\lambda f_x+\mu f_y+\nu f_z=0$, where 
  $$
  \lambda=\det\begin{pmatrix}\partial_y f^1&\partial_z f^1\cr \partial_y f^2&\partial_z f^2\end{pmatrix}
  \,,\;
  \mu=\det\begin{pmatrix}\partial_x f^1&\partial_z f^1\cr \partial_x f^2&\partial_z f^2\end{pmatrix}
  \,,\;
  \nu=\det\begin{pmatrix}\partial_x f^1&\partial_y f^1\cr \partial_x f^2&\partial_y f^2\end{pmatrix}.
  $$
  The three functions $\lambda,\mu,\nu$ are mutually functionally independent, in the sense that there exists no $F\in C^0(\R^3)$ such that 
  $F(\lambda,\mu,\nu)=0$ identically on any open set, but they are not {\em differentially} independent since $\partial_x\lambda+\partial_y\mu+\partial_z\nu=0$.
\end{example}
%
%
%
\begin{lemma}
  \label{lem:fi}
  Let $f:M\to\R^q$, $s_{n,r}<q<s_{n,r+1}$, be a generic map of full $(r+1)$-rank and set the $\lambda_a^{\alpha A}(f)$ as in Lemma~\ref{lem:v}. 
  Then, for any $s=0,1,\dots$, any $q$ distinct non-zero functions $\lambda_a^{\alpha A}(f)$ and their derivatives up to order 
  $s$ are functionally independent.
\end{lemma}
\begin{proof}
  To prove the claim of the lemma we need to consider the $j^s\lambda_a^{\alpha A}(f)$ as functions on $J^{s+r+1}(M,\R^q)$ and show that
  they are independent\footnote{Recall that some set of functions on a manifold is functionally independent if and only if
  their wedge product is not identically zero on any open set.} for a generic map $f$.

  First of all consider a single coefficient $\lambda_a^{\alpha A}(f)$ and let $\partial_{A'}f^{i_0}$ be one of the terms in it with 
  derivative of highest order (it can be either $|A'|=r$ or $|A'|=r+1$). Note that $i_0$ can be chosen as any integer between 1 and 
  $q$ since for every derivative order all components of $f$ with the same order enter in the expression of $\lambda_a^{\alpha A}(f)$. 
  Then $\partial_{A'B} f^{i_0}$ is one of the derivatives of highest order appearing in $\lambda_{aB}^{\alpha A}(f)$ and does 
  not appear in any other term $\lambda_{aB'}^{\alpha A}(f)$ with $|B'|\leq|B|$. This means that, in the expansion of
  $\bigwedge_{|B|\leq s}d\lambda_{aB}^{\alpha A}(f)$, it appears a single term proportional to $\bigwedge_{|B|\leq s}df^{i_0}_{A'B}$. The coefficient
  of this term is the product of some number of $f^{i_0}_A$ coordinates and therefore it is not identically zero on any open set, i.e. there is
  no differential relation between any $\lambda_a^{\alpha A}(f)$ and its derivatives.

  Now consider $q$ distinct non-zero coefficients $\{\lambda_{a_1}^{\alpha_1 A_1}(f),\dots,\lambda_{a_q}^{\alpha_q A_q}(f)\}$ and their derivatives
  up to order $B$. Then, for each $i=1,\dots,q$, there is a multi-index $A'_i$ such that $\bigwedge_{|B|\leq s}d\lambda_{a_iB}^{\alpha_i A_i}(f)$ 
  contains a non-zero term proportional to $\bigwedge_{|B|\leq s}df^{i}_{A'_iB}$. Moreover, since the $\lambda_{a_i}^{\alpha_i A_i}(f)$ are determinants 
  of matrices that differ by at least one column and whose columns are partial derivatives of $f$, each $\bigwedge_{|B|\leq s}df^{i}_{A'_iB}$ does
  not appear in $\lambda_{a_{i'}}^{\alpha_{i'} A_{i'}}(f)$ for $i'\neq i$. Hence $\bigwedge_{i=1,\dots,q}\left(\bigwedge_{|B|\leq s}d\lambda_{a_iB}^{\alpha_i A_i}(f)\right)$
  is not zero either because it contains $\bigwedge_{i=1,\dots,q}\bigwedge_{|B|\leq s}df^{i}_{A'_iB}\neq0$.
\end{proof}
%
%
\begin{theorem}
  \label{thm:Lf}
  Let $f:M\to\R^q$ be a generic map of full $r+1$-rank, $m=s_{n,r+1}-q$ and assume that 
  $q\geq m\sum_{i=0}^m {n-i\choose r}$ 
  and $m\leq (n-1)/(r+1)$.
  Then the associated order-$r$ compatibility PDO $\cL_f$ is surjective.
\end{theorem}
\begin{proof}
  Since $f$ is of full $r+1$-rank, at every point of $M$ are linearly independend all vectors $\partial_A f$, $|A|\leq r$ and 
  $s_{n,r+1}-s_{n,r}-m$ of the $\partial_A f$, $|A|=r+1$, 
  namely at every point hold $m$ linear dependence relations $\sum_{|A|\leq r+1}\lambda^{A}_a \partial_A f$. 
  Moreover, since $n\geq m+1+r m$, in some neighborhood of any point we can always rename coordinates so that 
  the components $\lambda^{\alpha_1\dots\alpha_{r+1} }_a(f)$ equal to zero have $\alpha_\ell>m+1$ for all $\ell=1,\dots,r+1$.

  At this point, if we restrict the corresponding compatibility PDO $\cL_f:C^r(M,\R^n)\to C^0(M,\R^m)$ to the linear subspace
  $h_{m+2}=\dots=h_n=0$, we get a UTS PDO (that we will keep calling $\cL_f$ to keep the notation light) of special kind, namely such that 
  the UTS condition is true for its components of {\em every} order, and whose 
  $m\sum_{i=0}^m {n-i\choose r}$ 
  components of order $r$ are all non-trivial. As explained in Remark~\ref{rm:sub}, the statements of Theorem~\ref{thm:UTSrel} 
  and~\ref{thm:UTS} also hold for this more special kind of operators and so there is a $s=\hat s(n,m,r)$ and a maximal algebraic 
  subbundle $\hat\cH^{r,s}_m(M)$ of $J^{r+s}(Hom(J^r(M,\R^q),J^0(M,\R^q)))$ such that the codimension of its complement is $n$ and
  $\cL_f$ is surjective if $j^{r+s}\Lambda_f(M)\subset\hat\cH^{r,s}_m(M)$. By hypothesis the components of $\cL_f$ are all 
  independent and so for, a generic $f$, $\cL_f$ is surjective.
\end{proof}
%

%
\section{Full rank isometries}
\label{sec:mr}
In the particular case of full 2-rank maps, Theorem~\ref{thm:Lf} has the following important consequences
about isometric immersions. 
Given $n=1,2,\dots$, denote by $m_n$ the middle root of $n+s_n-m=m(m+1)(n+1)-ms_m$. 
%
\begin{definition}
  Let $s(n,m,r)$ and $$\cH^{1,s}_m(M)\subset J^{1+s}\left(Hom\left(J^1(M,\R^q),J^0(M,\R^q)\right)\right),\;q=n+s_n-m,$$
  be the function and the subbundle in the claim of Theorem~\ref{thm:UTSrel}. 
  For every $s\geq s(n,m,1)$, we denote by $F^{s+3}(M,\R^q)\subset C^{s+3}(M,\R^q)$ the open set of full 2-rank maps $f$ such that $j^{s+1}\Lambda_f(M)\subset\cH^{1,s}_m(M)$.
\end{definition}
\begin{proposition}
  \label{prop:dense}
  The set $F^{s+3}(M,\R^q)$ is dense in the set of $C^{s+3}$ full 2-rank maps for $q\geq  n+s_n-m_n$. 
\end{proposition}
\begin{proof}
  The condition $q\geq n+s_n-m_n$ is equivalent to $q\geq m\sum_{i=0}^m {n-i\choose 1}$ and $m_n\leq(n-1)/2$ for all $n\geq2$.
  Hence the proof of Theorem~\ref{thm:Lf} shows that, 
  under such hypotheses, a generic $f$ will induce a compatibility UTS PDO $\cL_f$ such that $j^{1+s}\Lambda_f(M)$ will be entirely contained in $\cH^{1,s}_m(M)$.
\end{proof}
\begin{theorem}
  \label{thm:infinv}
  For $q\geq  n+s_n-m_n$, there is a $d=d(n,m)\geq nm+3$, $m=n+s_n-q$, such that the metric inducing operator $\cD_{M,q}$ admits an infinitesimal 
  inversion of order 0 and defect $d$ over $F^d(M,\R^q)$.
\end{theorem}
\begin{proof}
  By Proposition~\ref{prop:dense}, $\cL_f$ is surjective for every $f\in F^d(M,\R^q)$. By Theorem~\ref{thm:UTSrel}, the 
  order of its $\cH^{1,s}_m(M)$-universal right inverse $\cM_s$ is at least $\bar s=s(n,m,1)\geq mn$. Since under such conditions 
  $\cL_f$ is surjective, we can choose the $h_\alpha$ in system~(\ref{eq:NashTrick}) so that the compatibility condition~(\ref{eq:Lf(h)=g}) 
  is satisfied. In order to solve~(\ref{eq:NashTrick}) this way, we need the full 2-rank map $f$ to be at least of class $C^{\bar s+1+2}$, 
  since we must be able to evaluate $j^{s+1}\Lambda_f(M)$ and the components of $\cL_f$ are polynomials in the components of $j^2f$.
  Hence the infinitesimal inverse operator of $\cD_{M,q}$ has order 0 and defect $d=\bar s+3\geq nm+3$ over $F^d(M,\R^q)$.
\end{proof}
From this last result and Gromov's Theorem~\ref{thm:IFT} we get the following two results:
\begin{theorem}
  \label{thm:isom}
  Let $q\geq n+s_n-m_n$ and $d=d(n,m)$. If $f_0\in F^{d}(M,\R^q)\cap C^\infty(M,\R^q)$ and $g_0=\cD_{M,q}f_0$, then there is a 
  $C^{d+1}$-neighborhood of zero $U\subset\Gamma^{d+1}(S_2^0(M))$ such that, for all $g\in U\cap\Gamma^{\sigma}(S_2^0(M))$, $\sigma\geq d+1$, 
  the metric $g_0+g$ can be induced by an immersion $f\in C^{\sigma}(M,\R^q)$.
\end{theorem}
\begin{corollary}
  \label{cor:open}
  For $q\geq n+s_n-m_n$, the metric inducing operator $\cD_{M,q}$ is open over the set $F^{\infty}(M,\R^q)=F^{d}(M,\R^q)\cap C^\infty(M,\R^q)$, $d=d(m,n)$.
\end{corollary}
\begin{remark}
  Notice that the condition $q\geq n+s_n-\sqrt{n/2}-1/2$ is weaker but almost equivalent to $q\geq n+s_n-m_n$ since $m_n\leq \sqrt{n/2}\leq m_n+1/2$.
\end{remark}
In Examples~\ref{ex:frEuc} and~\ref{ex:frTor} we pointed out that full-rank maps arise on Euclidean spaces for all allowed combinations of $n$, $m$ and $q$ and in some 
particular case on tori. We conclude the article by briefly discussing the existence of maps of full rank on general open manifolds (namely manifolds without compact components).
\begin{theorem}[Gromov, 1986]
  \label{thm:om}
  Let $E$ be an open smooth manifold, $F\to E$ be a smooth fibration over it and $\cR\subset J^rF$ an open subset invariant by the action of $\Diff(E)$.
  Denote by $\Hol^r(\cR)\subset\Gamma^0\cR$ the set of {\em holonomic} sections of $\cR\to E$, namely sections $\phi=j^rf$ for some $f:E\to F$.
  Then $\cR$ satisfies the {\em parametric h-principle}, namely the map $j^r:\Hol^r(\cR)\to\Gamma^0\cR$ is a weak homotopy equivalence. 
\end{theorem}
\begin{remark}
  We recall that, in particular, this means that $j^r$ induces an isomorphism between the first homotopy groups $\pi_0(\Hol^r(\cR))$
  and $\pi_0(\Gamma^0\cR)$. The surjectivity expresses the {\em h-principle} for $\cR$, namely that every section $E\to\cR$ can be
  homotoped, in $\Gamma^0\cR$, to the $r$-jet of some section $f:E\to F$. The injectivity expresses the {\em 1-parametric h-principle}, namely
  that every two holonomic sections $j^rf_1$, $j^rf_2$ that are homotopic in $\Gamma^0\cR$ are also homotopic in $\Hol^r(\cR)$. 
  See~\cite{Gro86}, Sec. 1.2.1 (C), and~\cite{EM02}, Sec. 6.2, for more details about multi-parametric and other flavors of h-principles.
\end{remark}
Consider the subbundle $\cF^r(M,\R^q)\subset J^r(M,\bR^q)$, with $s_{n,r-1}<q<s_{n,r}$, whose fibers over each point $x$ are the subset 
such that $(f^i_A)$, $|A|\leq r$, seen as a $q\times s_{n,r}$ matrix, has rank $q$. Then $\Hol^r(\cF^r(M,\R^q))=F^r(M,\R^q)$. 
Clearly $\cF^r(M,\R^q)$ is open and invariant by all diffeomorphisms of the base, so that we can claim the following result:
\begin{corollary}
  \label{cor:h-Princ}
  Let $M$ be an open manifold. Then $F^r(M,\R^q)\neq\emptyset$ if and only if there exist $C^0$ sections $M\to\cF^r(M,\R^q)$.
  In particular full $r$-rank maps $M\to\R^q$ arise on any parallelizable open manifold $M$ for all $s_{n,r-1}<q<s_{n,r}$.
\end{corollary}
%
%
\bibliography{refs}

\providecommand{\bysame}{\leavevmode\hbox to3em{\hrulefill}\thinspace}
\providecommand{\MR}{\relax\ifhmode\unskip\space\fi MR }
\providecommand{\MRhref}[2]{%
  \href{http://www.ams.org/mathscinet-getitem?mr=#1}{#2}
}
\providecommand{\href}[2]{#2}
\begin{thebibliography}{{DeL}11}

\bibitem[{DeL}10]{DeL10}
R.~{DeLeo}, \emph{A note on non-free isometric immersions}, Russian
  Mathematical Surveys \textbf{65} (2010), no.~3, 577--579.

\bibitem[{DeL}11]{DeL11}
R.~{DeLeo}, \emph{Solvability of the cohomological equation for regular vector
  fields on the plane}, Annals of Global Analysis and Geometry \textbf{39:3}
  (2011), 231--248.

\bibitem[DeL15]{DeL15c}
R.~DeLeo, \emph{Weak solutions of the cohomological equation in the plane for
  regular vector fields}, Mathematical Physics, Analysis and Geometry
  \textbf{18} (2015), no.~18.

\bibitem[DL07]{DL07}
G.~D'Ambra and A.~Loi, \emph{Non-free isometric immersions of {R}iemannian
  manifolds}, Geom. Dedicata \textbf{127} (2007), 65--88.

\bibitem[EM02]{EM02}
Y.~Eliashberg and N.~Mischachev, \emph{Introduction to the $h$-principle}, GSM,
  vol.~48, AMS, 2002.

\bibitem[For95]{For95}
G.~Forni, \emph{The cohomological equation for area-preserving flows on compact
  surfaces}, Electronic Research Announcements of the AMS \textbf{1} (1995),
  no.~3.

\bibitem[For99]{For99}
\bysame, \emph{Solutions of the cohomological equation for area-preserving
  flows on compact surfaces of higher genus}, Ann. of Math. \textbf{2:2}
  (1999), 295--344.

\bibitem[Gro86]{Gro86}
M.~Gromov, \emph{Partial {D}ifferential {R}elations}, Springer Verlag, 1986.

\bibitem[Gro15]{Gro15}
\bysame, \emph{Geometric, algebraic and analytic descendants of {N}ash
  isometric embedding theorems},
  \url{http://www.ihes.fr/~gromov/PDF/nash-copy-Oct9.pdf}, 2015.

\bibitem[Nas56]{Nas56}
J.~Nash, \emph{The imbedding problem for {R}iemannian manifolds}, Ann. of Math.
  \textbf{63:1} (1956), 20--63.

\bibitem[Nov08]{Nov08}
S.P. Novikov, \emph{Dynamical systems and differential forms. {L}ow dimensional
  {H}amiltonian systems}, Geometric and probabilistic structures in dynamics,
  Contemp. Math., vol. 469, AMS, 2008, pp.~271--287.

\end{thebibliography}
\end{document}